\documentclass[11pt]{amsart}
\usepackage{xcolor} 
\usepackage[utf8]{inputenc}
\usepackage{lmodern}
\usepackage{amsmath, amsthm, amssymb, amsfonts,bm}
\usepackage[normalem]{ulem}
\usepackage{hyperref}

\usepackage{verbatim} 
\usepackage{longtable}

\usepackage{mathtools}

\usepackage{tikz}
\usetikzlibrary{decorations.pathmorphing}
\tikzset{snake it/.style={decorate, decoration=snake}}
\usepackage{tikz-cd}
\usetikzlibrary{positioning}

\usepackage{caption}

\usepackage{tikz-cd}
\usetikzlibrary{arrows}

\theoremstyle{plain}

\newtheorem{theorem}{Theorem}[section]

\newtheorem{definition}[theorem]{Definition}

\newtheorem{prop}[theorem]{Proposition}
\theoremstyle{definition}
\newtheorem{exmp}[theorem]{Examples}

\newtheorem{cor}[theorem]{Corollary}
\theoremstyle{remark}
\newtheorem{rem}[theorem]{Remark}
\newtheorem{lem}[theorem]{Lemma}

\DeclareFontFamily{OT1}{rsfs}{}
\DeclareFontShape{OT1}{rsfs}{n}{it}{<-> rsfs10}{}
\DeclareMathAlphabet{\curly}{OT1}{rsfs}{n}{it}

\def\Z{\mathbb{Z}}

\def\C{\mathbb{C}}

\def\P{\mathbb{P}}

\def\O{\mathcal{O}}

\def\a{\alpha}

\def\m{\mu}

\def\m{\mathcal}

\def\Spec{\textrm{Spec}}

\def\dim{\textrm{dim}}

\def\ker{\textrm{ker}}

\def\zm{\mathbb{Z}/m\mathbb{Z}}
\def\wA{\widehat{A}}
\def\wAm{\widehat{A}(m)}

\def\QCoh{\operatorname{QCoh}}
\def\Mod{\operatorname{Mod}}
\def\mod{\operatorname{mod}}
\def\tC{\widetilde{C}}
\def\tY{\widetilde{Y}}

\newsavebox{\leftbox} \newsavebox{\rightbox}%

\NewDocumentCommand{\lrboxbrace}{s O{\{} O{\}} O{0.05\linewidth} m O{0.5\linewidth} m}{
  \begin{lrbox}{\leftbox}
    \IfBooleanTF{#1}
      {\begin{varwidth}{#4}#5\end{varwidth}}
      {\begin{minipage}{#4}#5\end{minipage}}
  \end{lrbox}
  \begin{lrbox}{\rightbox}
    \IfBooleanTF{#1}
      {\begin{varwidth}{#6}#7\end{varwidth}}
      {\begin{minipage}{#6}#7\end{minipage}}
  \end{lrbox}
  \ensuremath{\usebox\leftbox\left\{\,\usebox\rightbox\,\right\}}
}

\usepackage[backend=biber,style=alphabetic,sorting=nyt, giveninits=true,doi=false,isbn=false,maxbibnames=99]{biblatex} 
\usepackage{vmargin}
\setpapersize{A4}
\setmarginsrb{27mm}{12mm}{27mm}{12mm}%
{12mm}{10mm}{5mm}{10mm}

\usepackage{tikz}
\usepackage{lmodern}
\usetikzlibrary{decorations.pathmorphing}

\begin{document}
\title{Spectral correspondence for cyclic Higgs bundles}
\author{JIA CHOON LEE and Ana Pe\'on-Nieto}

\address{Institute of Geometry and Physics, University of Science and Technology of China, 96 Jinzhai Road, Hefei 230026 P.R. China}
\email{jiachoonlee@ustc.edu.cn,jiachoonlee@outlook.com}
\address{Facultad de Ciencias Matemáticas, Universidade de Santiago de Compostela, Rúa de Lope Gómez de Marzoa s/n, 15782 Santiago de Compostela, SPAIN,
      \newline \textit{and}
      \newline\indent School of Mathematics,  University of Birmingham, Watson Building, Edgebaston, Birmingham B15 2TT, UK}
\email{ana.peon@usc.es, a.peon-nieto@bham.ac.uk}


\begin{abstract}
In this paper, we describe the spectral correspondence for cyclic Higgs bundles from the viewpoint of quiver bundles. Under this framework, we establish a one-to-one correspondence between cyclic Higgs bundles on a curve and sheaves on a noncommutative surface whose noncommutative structure originates from the path algebra associated to the cyclic quiver. As applications, this correspondence generalizes the known spectral correspondence for $U(p,p)$-Higgs bundles and establish a connection between the spectral data for $U(p,q)$-Higgs bundles and modules over the sheaf of even Clifford algebras of a conic fibration. 
\end{abstract}

\baselineskip=14.5pt
\maketitle

\setcounter{tocdepth}{1}
\tableofcontents

\section{Introduction}
A cyclic Higgs bundle on a smooth projective curve $C$ is a $GL(r, \C)$-Higgs bundle consisting of a vector bundle $E= \bigoplus_{i\in \zm} E_i$ and a Higgs field $\Phi: E\to E\otimes K_C$ of the form
            \begin{equation}\label{eq:associated bundle}
                \Phi= \begin{pmatrix}
                            0&0&\cdots&0&\phi_{m-1}\\
                            \phi_0&0&&&0\\
                            0&\phi_1&&&0\\
                            \vdots &&\ddots&&\vdots\\
                            0&\cdots&&\phi_{m-2}&0
                \end{pmatrix}
            \end{equation}
where $\phi_i: E_i\to E_{i+1}\otimes K_C$ for $i\in \zm$.

This notion arose originally in the work of Simpson \cite{Simpson-middleconvolution} (called cyclotomic Higgs bundles) and was defined as the fixed points for the action of $\zm$ on the Higgs field. From the perspective of moduli spaces, the moduli space of cyclic Higgs bundles can be viewed as the fixed point locus of an $\zm$-action on the moduli space of $GL(r,\C)$-Higgs bundle. Cyclic Higgs bundles have also been investigated in other contexts. When $m=2$, they are Higgs bundles for real Lie group $U(p,q)$ and the topology of their moduli spaces have been studied extensively using Morse theory by Bradlow, Garc\'{i}a-Prada and Gothen \cite{bradlow-gp-gothen}. Analytically, the work of Baraglia \cite{baraglia} establishes a correspondence between cyclic Higgs bundles and solutions of the affine Toda equation. This is later generalized to non-compact surfaces by Li and Mochizuki \cite{Li-Mochizuki2}. More recently, the concept of cyclic Higgs bundles has been generalized to arbitrary reductive groups through the Lie-theorectic perspective of Vinberg pairs by Garc\'{i}a-Prada and Gonz\'{a}les \cite{gp-vinberg,Garc_a_Prada_2026} (see also \cite{garcia-ramanan}).

In this paper, we are interested in studying the spectral correspondence for cyclic Higgs bundles. The classical spectral correspondence for $GL(r,\C)$-Higgs bundles is  originally established by Hitchin \cite{hitchin-stable}, and subsequently generalized by Beauville--Narasimhan--Ramanan (BNR) \cite{bnr} and Simpson \cite{simpsonmoduli}. Roughly speaking, the original viewpoint of Hitchin converts a Higgs bundle into a line bundle on its corresponding spectral curve (when smooth) by taking a fiberwise eigenspace decomposition. On the other hand, the viewpoint of BNR treats Higgs bundles as modules over the symmetric algebra $S^\bullet K_C^{-1}$ which then lift to torsion-free sheaves on the spectral curve (when integral). Simpson further interprets these modules as pure dimension one coherent sheaves on the surface $X=\underline{\Spec}(S^\bullet K_C^{-1})$. The two approaches of Hitchin and BNR/Simpson parallel an elementary equivalence in linear algebra: a vector space equipped with an endomorphism is equivalent to a module over the polynomial ring $\C[t]$. This can be naturally generalized to an equivalence in the theory of quiver that we will use: a representation of a quiver is equivalent to a module over its path algebra.

A cyclic Higgs bundle can be seen as a quiver bundle for the cyclic quiver $Q(m)$, or equivalently, a representation of $Q(m)$ in the category of vector bundles.
\begin{equation*}
                        Q(m)= 
                    \begin{tikzcd}
              v_0\arrow[r,"a_0"]&v_1\arrow[r,"a_1"]&... \arrow[r]&v_{m-1}\arrow[lll,bend right,"a_{m-1}"]
                    \end{tikzcd}
            \end{equation*}
We will adopt the module-theoretic viewpoint which is more flexible when the spectral curves become singular. This means that we first treat a cyclic Higgs bundle as a module over the sheaf of the associated path algebra $A(m)$ (generalizing the symmetric algebra $S^\bullet K_C^{-1}$) on $C$. This path algebra $A(m)$ is always noncommutative and of infinite rank. By performing a central reduction of $A(m)$, we obtain a sheaf of noncommutative algebras $\wAm$ of finite rank on the surface $Y = \underline{\Spec}(S^\bullet K_C^{-m})$. Then we can lift the $A(m)$-modules on $C$ to $\wAm$-modules on $Y$. By rewriting the $\wAm$-modules in terms of quiver sheaves, we arrive at the following version of spectral correspondence for cyclic Higgs bundles. 
\begin{theorem}[Corollary~\ref{cor:equiv5}]\label{thm:main}
    Fix a dimension vector $\bm{p}=(p_0, \dots, p_{m-1})$ and a degree vector $\bm{d}=(d_0, \dots, d_{m-1})$.
Let $(\tC_0, \dots,\tC_{m-1})\in B(\bm{p})$ be a fixed collection of spectral curves in $Y$. The following two categories are equivalent:
\begin{enumerate}
    \item (Higgs data) The category of cyclic Higgs bundles $(\{E_i\}_{i\in \zm}, \{\phi_i\}_{i\in \zm})$ on $C$ with dimension vector $\bm{p}$ and degree vector $\bm{d}$, such that the spectral curve of  $\Phi_i:=\phi_{i-1} \circ \dots \circ \phi_{i+1} \circ \phi_i $ is $\tC_i$ for all $i\in \zm$.
    
    \item (Spectral data) The category of $M_Y$-twisted $Q(m)$-quiver sheaves $(\{F_i\}_{i\in\zm}, \{\psi_i\}_{i\in\zm})$ on $Y$ satisfying:
    \begin{itemize}
        \item[(a)] For each $i \in \mathbb{Z}/m\mathbb{Z}$, $F_i$ is a pure dimension one sheaf on $Y$ that is scheme-theoretically supported on $\tC_i$ (i.e. an $\O_{\tC_i}$-module) with Euler characteristic $\chi(F_i)= d_i + p_i(1-g)$. 
        \item[(b)] The relation $\psi_{i-1}\circ\cdots \circ \psi_{i} = \operatorname{Id}_{F_i}\otimes_{\O_Y} \lambda $ holds in $\operatorname{Hom}_{\O_Y}(F_i\otimes_{\O_Y} \pi^*K_C^{-m}, F_i )$ for all $i\in \mathbb{Z}/m\mathbb{Z}$. 
    \end{itemize}
\end{enumerate}
\end{theorem}
\begin{rem}
As we will establish in Proposition~\ref{prop:equiv3}, the spectral data above is equivalent to an $\wAm$-modules. Therefore, Theorem~\ref{thm:main} can be stated as a one-to-one correspondence between cyclic Higgs bundles on $C$ and coherent sheaves on the noncommutative surface $(Y,\wAm)$. 
\end{rem}
This equivalence can be seen as a natural generalization of the works of BNR/Simpson in the case of $GL(r,\C)$,  Schaposnik \cite{schaposnik14} in the case of $U(p,p)$ (which parallels the approach of Hitchin) and Pe\'{o}n-Nieto \cite{peónnieto2015cameraldatasup1phiggsbundles} in the case of $SU(p+1,p)$. There is also a general approach to abelianization of Higgs bundles for arbitrary quasi-split real groups via cameral cover that holds for regular locus by the work of Pe\'{o}n-Nieto and Garc\'{i}a-Prada \cite{gp-peon-nieto}. In Section~\ref{sec:examples}, we study some cases where a further reduction of the spectral data is posssible. 

A further advantage of working on the noncommutative surface $(Y, \wAm)$ is that it reveals a surprising relation to another noncommutative surface studied in the context of the derived category of a conic fibration. More precisely, we show that the sheaf of algebras $\wAm$ is isomorphic to the sheaf of even Clifford algebras associated to a conic fibration over $Y$ (Proposition~\ref{prop:clifford}). The noncommutative algebraic varieties with the sheaf of even Clifford algebras $\mathcal{C}\ell_0$ as structure sheaf appear originally in the work of Kuznetsov on homological projective duality \cite{kuznetsov-quadric}. For example, for a vector space $W$, one of the main results in loc. cit. is that the homological projective dual of $\P(W)$ in the double Veronese embedding $\P(W)\to \P(S^2W)$ is the noncommutative algebraic variety $(\P(S^2 W), \mathcal{C}\ell_0)$. 

The idea of encoding spectral data in terms of sheaves on surfaces with extra structure has also been explored for other variants of Higgs bundles. 
Notably, the use of Koszul duality and noncommutative geometry on some surfaces to establish spectral correspondence was developed by Katzarkov, Orlov and Pantev for twisted and curved Higgs bundles \cite{KOP}. On the other hand, in the context of parabolic Higgs bundles, one can establish a spectral correspondence in terms of blow-up of ruled surface (see \cite{lee2025relativespectralcorrespondenceparabolic}). Further blow-up patterns are required if one is interested in parabolic Higgs bundles with irregular singularities \cite{Kontsevich-Soibelman,Szab2017TheBG,Diaconescu_2018}. Alternatively, one can encode the parabolic structure in terms of root stacks and orbifold surfaces \cite{chuang}. 

Finally, we conclude the paper in Section~\ref{sec:future} by exploring several future directions and open questions stemming from this spectral correspondence.

\subsection{Conventions}
Throughout this paper, we work over the field of complex numbers $\C$. Moreover, all modules over an algebra are assumed to be right modules unless explicitly stated otherwise. 

\subsection{Acknowledgements}{A.P.N. was funded through the schemes Proyectos de Consolidación Investigadora (grant nº CNS2022-136042), Proyectos de generación del conocimiento (nº PID2023-147785NA-I00) and Proyectos de colaboraci\'on internacional (nº PCI2024-155049-2). She also acknowledges support from the COVID Support Programme from the Strategic Project Office of the University of Birmingham.}

\section{Quiver sheaves and modules over the sheaf of path algebras}
\subsection{Quiver Sheaves}
In this section, we recall the framework of quiver sheaves and modules over the associated sheaf of path algebras, following \cite{CGP}. While we focus entirely in the case of curves, this framework is naturally defined over any algebraic variety. 

Let $Q=(Q_0,Q_1)$ be a fixed quiver where $Q_0$ and $Q_1$ are the sets of vertices and arrows together with the head and tail maps:
\begin{equation*}
    h,t:Q_1\to Q_0.
\end{equation*}
Let $C$ be a smooth projective curve and $M=\{M_a\}_{a\in Q_1}$ be a collection of line bundles on $C$. An $M$-twisted $Q$-quiver sheaf is a pair $(\m E, \Phi)$ consisting of a collection of coherent $\O_C$-modules $\m E= \{E_i\}_{i\in Q_0}$ and a collection of morphisms $\Phi =\{\phi_a \}_{a\in Q_1}$ where
\begin{equation*}
    \phi_a : E_{t(a)} \otimes M_a  \to E_{h(a)}.
\end{equation*}
When $E_i$ are all locally free, we will call it an $M$-twisted $Q$-quiver bundle and denote by $\bm{p} = (\dim(E_i))_{i\in Q_1}$ the dimension vector of the quiver bundle. 

\begin{exmp}
A $GL(r,\C)$-Higgs bundle $(E, \phi: E\to E\otimes K_C )$ on $C$ is an $M$-twisted $Q$-quiver bundle with dimension vector $\bm{p}=(r)$, $M= \{K_C^{-1}\}$, and
\begin{equation*}
    Q =\begin{tikzcd}
        v\arrow[loop]
    \end{tikzcd}
\end{equation*}    
\end{exmp}

\begin{exmp}
A $U(p,q)$-Higgs bundle on $C$ is an $M$-twisted $Q$-quiver bundle with dimension vector $\bm{p}=(p,q)$, $M=\{M_{a_1}=K_C^{-1},M_{a_2}=K_C^{-1}\}$ and
\begin{equation*}
    Q= \begin{tikzcd}
        v_1\arrow[r,bend right
        ,"a_1"]&v_2\arrow[l,bend right,"a_2"]
    \end{tikzcd}
\end{equation*}
\end{exmp}
\begin{exmp}
    A holomorphic chain of length $k$ is an $M$-twisted $Q$-quiver bundle where $M_{a_i}=K_C^{-1}$ for all arrows $a_i\in Q_1$ and
    \begin{equation*}
        Q = \begin{tikzcd}
        v_1\arrow[r,"a_1"] &v_2 \arrow[r,"a_2"]&\dots \arrow[r,"a_{k-1}"]&v_k
        \end{tikzcd}
    \end{equation*}
    They appear as the fixed points of the $\C^*$-action on the moduli space of $GL(r,\C)$-Higgs bundles (corresponding to complex variations of Hodge structure). 
\end{exmp}

            \subsection{Path algebras}
Recall that for each quiver $Q=(Q_0,Q_1)$, a path (of length $m$) is a sequence $a_1\cdots a_k$ of arrows $a_i\in Q_1$ such that $t(a_{i})= h(a_{i+1})$ for $i=1,\dots , k-1$. For each vertex $i\in Q_0$, we define the trivial path $e_i$ (of length 0). 

Let $S= \bigoplus_{i\in Q_0} \O_C\cdot e_i$ be a commutative $\O_C$-algebra generated by the formal symbols $e_i$ with relations 
\[\begin{cases}
    e_i\cdot e_{i'} = e_{i}, \quad \textrm{if }i=i' \\
    e_i\cdot e_{i'} = 0 , \quad \textrm{otherwise}. 
\end{cases}\]
for each $i,i'\in Q_0$.
Moreover, let $\widetilde{M}= \bigoplus _{ a\in Q_1} M_a$ and define a $S$-bimodule structure on it as follows: the left module structure is given by
\[\begin{cases}
    e_i\cdot m = m, \quad \textrm{if }m\in M_a\textrm{ and }i=h(a)\\
    e_i\cdot m=0, \quad \textrm{otherwise}
\end{cases}  \]
for all $i\in Q_0, a\in Q_1, m\in M_a$, and likewise for the right module structure. Following \cite{CGP}, we define the sheaf of $M$-twisted path algebra $A$ of $Q$ as the tensor $S$-algebra of the $S$-bimodule $\widetilde{M}$:
\[ A = \bigoplus _{\ell\geq 0} \widetilde{M} ^{\otimes_S \ell}\]
where we view $\widetilde{M}^{\otimes 0} := S$. 

Equivalently, one can directly define an $\O_C$-algebra structure on 
\[ MQ:= \bigoplus_{\textrm{all paths }p}M_{p} \]
where $M_p = M_{a_1}\otimes ... \otimes M_{a_k}$ and the relation is given by 
\[\begin{cases}
   m_p\cdot m_{p'} = m_{p}\otimes m_{p'}, \quad \textrm{if }t(p) = h(p') \\
    m_p\cdot m_{p'} = 0, \quad \textrm{otherwise }
\end{cases}\]
where $m_p\in M_p,m_{p'}\in M_{p'}$ for all paths $p,p'$. Then it is straightforward to see that the $A\cong MQ$ as $\O_C$-algebra e.g. the algebra $A$ has relations such as $m_a \otimes m_{a'} = m_a \otimes e_{h(a')} m_{a'} = m_a \cdot e_{h(a')} \otimes m_a$ which vanishes when $t(a)\neq h(a')$. The unit element in $A$ admits a decomposition into orthogonal idempotents
\begin{equation}\label{eq:decomposition}
    1_A = \sum_{i\in Q_0}e_i. 
\end{equation} 

Just as the category of quiver representations is equivalent to the category of modules over the associated path algebra, there is an analogue for quiver sheaves. Fix a quiver $Q$, a coherent right module over the associated sheaf of path algebra $A$ (or simply, $A$-modules) is an $\O_C$-module $E$ equipped with a right $A$-module structure given by an $\O_C$-module morphism $E\otimes_{\O_C} A\to E$ satisfying the axioms of right modules over an algebra.  

\begin{prop}\label{prop:equiv1}
There is an equivalence between the category of $M$-twisted $Q$-quiver bundles on $C$ and the category of coherent right $A$-modules $\mod_C\!\text{-}A$.
\end{prop}
\begin{proof}
    Given a $M$-twisted $Q$-quiver bundles $(\{E_i\}_{i\in Q_0},\{\phi_a\}_{a\in Q_1})$, we associate to it $E = \bigoplus_{i\in Q_0} E_i$ with a right module structure $E\otimes_{\O_C} A\to E$ induced by the maps $\phi_a$ which are generators of the path algebra $A$. This construction defines a functor in the forward direction. Conversely, given a right $A$-module $\m E$ equipped with an $A$-module structure $\mu_A: \m E\otimes_{\O_C}A\to \m E$, the decomposition of the unit $1_A$ induces a decomposition of $\m E = \oplus _{i\in Q_0}E_i$. Restricting $\mu_A$ to $\m E\otimes_{\O_C} \widetilde{M}$ induces the arrow morphisms $\phi_a: M_a\otimes E_{t(a)} \to E_{h(a)}$ for $a\in Q_1$.  For details, we refer the readers to \cite[Proposition 5.2]{CGP}. 
\end{proof}

\section{Spectral correspondence for cyclic Higgs bundles}
\begin{definition}
                    A cyclic Higgs bundle on $C$ with dimension vector $(p_0,\dots,p_{m-1})$ is a $M$-twisted $Q(m)$-quiver bundle with dimension vector $(p_0,\dots,p_{m-1})$, $M=\{M_{a_i} = K_C^{-1}\}_{i\in \Z/m\Z}$ and 
                    \begin{equation*}
                        Q(m)= 
                    \begin{tikzcd}
              v_0\arrow[r,"a_0"]&v_1\arrow[r,"a_1"]&... \arrow[r]&v_{m-1}\arrow[lll,bend right,"a_{m-1}"]
                    \end{tikzcd}
            \end{equation*}
                \end{definition}
\begin{rem}
Clearly, a cyclic Higgs bundle in the definition is equivalent to a $GL(\sum p_i, \C)$-Higgs bundles whose Higgs field is of the form \ref{eq:associated bundle} stated in the introduction. 
\end{rem}

    Let $A(m)$ be the sheaf of $M$-twisted path algebras on $C$ associated to the quiver $Q(m)$. As $A(m)$ is noncommutative, we begin by studying its central subalgebra $Z(A(m))$ which is commutative. Note that we have the natural direct sum decomposition
                \begin{equation}\label{path decomposition}
                    A(m) = \bigoplus_{ i,j\in Q_0}A_{i\leftarrow j}
                \end{equation}
      where $A_{i\leftarrow j}:= e_iA(m)e_j$ corresponds to paths from $v_j$ to $v_i$. 

    \begin{lem}\label{central1}
    $Z(A(m))\subset \bigoplus_{i\in \zm} A_{i\leftarrow i}\subset A(m)$.
    \end{lem}
    \begin{proof}
        Indeed, if $z\in Z(A(m))$ and we write $z= \sum_{i,j\in Q_0} z_{i\leftarrow j}$ according to the decomposition \ref{path decomposition}, then the commutativity $e_k\cdot z = z\cdot e_k$ implies that 
        \[  \sum_{j\in Q_0} z_{k\leftarrow j} = \sum_{i\in Q_0} z_{i\leftarrow k} \]
        It follows that only $z_{k\leftarrow j}=0$ for $k\neq j$. 
    \end{proof}
For each vertex $i\in \zm$, we define the path (loop) of length $m$ 
\[c_i= a_{i-1}\cdot a_{i-2} \dots a_{i+1}\cdot a_{i}\]
where all subscripts are understood modulo $m$. Clearly, we have 
\[ A_{i\leftarrow i }= \bigoplus_{\ell=0}^\infty M_{c_i^\ell}= \bigoplus_{\ell=0}^\infty M_{c_i}^{\otimes \ell} \]
as all paths starting and ending at the vertex $v_i$ must be a composition of the loop $c_i$. Since $M_{c_i}= K_C^{-m}$, it induces a canonical isomorphism $\iota_i: S^{\bullet}(K_C^{-m}) \xrightarrow{\sim} A_{i\leftarrow i}$. Let
\begin{align}
    \Delta : S^\bullet (K_C^{-m}) &\to  \bigoplus_{i\in \zm} A_{i\leftarrow i}\subset A(m),\label{eq:diagonal} \\
    s &\mapsto \sum_{i\in \zm}\iota_i(s) \nonumber
\end{align}
be the diagonal map and denote its image by $A_\Delta\subset A(m)$. 
\begin{prop}\label{central2}
The central subalgebra $Z(A(m))$ is equal to $A_{\Delta}$. Consequently,
$$ Z(A(m)) \cong S^\bullet(K_C^{-m}) $$
as $\O_C$-algebras. 
\end{prop}
\begin{proof}
    Let $z= \sum_{i\in \zm}z_i\in Z(A(m))$ where $z_i\in A_{i\leftarrow i}$ by Lemma~\ref{central1}. It remains to impose that $z$ commutes with all local sections $\alpha_i\in M_{a_i}$. By matching the idempotents (heads and tails), the condition $z\cdot \alpha_i = \alpha_i \cdot z$ implies that 
    \begin{equation}\label{equality}
        z_{i+1}\cdot \alpha_i = \alpha_i\cdot z_{i}
    \end{equation}
    which holds in $A_{i+1\leftarrow i}$. If $z_i\in M_{c_i^\ell}$, then the equality forces $z_{i+1}\in M_{c_{i+1}^\ell}.$ As paths, we have $c_{i+1}^\ell \cdot a_i = a_i \cdot c_{i}^{\ell}$ which induces a canonical isomorphism $M_{c_{i+1}}^{\otimes \ell}\otimes M_{a_i} \cong M_{a_i}\otimes M_{c_i}^{\otimes \ell}$. Since $M_{a_i}$ is a line bundle (locally isomorphic to $\O_C$), the equality \ref{equality} holds for all local sections $\alpha_j$ implies that the local sections $z_{i+1}$ and $z_i$ are identified under the isomorphism $\iota_i \circ \iota_{i+1}^{-1}$. Since this holds for all $i$, all components $z_i$ are the image of a single section $s\in S^\bullet(K_C^{-m})$, so $Z(A(m))\subset A_\Delta$. 

    Conversely, it is straightforward to verify that any diagonal element $s\in A_{\Delta}$ commutes with the idempotents $e_i$ and the arrow generators of $A(m)$ to establish the reverse direction. 
\end{proof}

By Proposition \ref{central2}, we know $A(m)$ is a $S^\bullet(K_C^{-m})$-module. It is then natural to consider the total space of $K_C^m$:
\[\pi: Y :=\underline{\operatorname{Spec}}(S^\bullet (K_C^{-m}))\to C\]
Let 
\[A_Y(m)= \pi^*A(m)\]
be the usual pullback of a coherent module which is also the sheaf of $M_Y$-twisted path algebras on $Y$ of the quiver $Q$ with twisting line bundles $M_Y: = \{M_{Y,a_i}= \pi^* (K_C^{-1})\}_{i\in \zm}$. Consider the two morphisms of sheaves: 
\begin{itemize}
    \item The composition $K_C^{-m} \hookrightarrow S^\bullet K_C^{-m} \xrightarrow{\Delta} A(m)$ is a morphism of $\O_C$-modules. We denote its pullback by $\nu_1 := \pi^*(\Delta|_{K_{C}^{-m}})$. Locally, for a section $s\in K_C^{-m}$, this acts as: 
    \[ \nu_1(\pi^*s) =\pi^* \left(\Delta(s)\right) = \sum_{i\in \zm}\pi^*(\iota_i(s))\] 
    
    \item Let $\lambda: \pi^*K_C^{-m} \to \O_Y$ be the morphism induced by the tautological section which lies in $H^0(Y, \pi^*K_C^m)$. We define
    \[ \nu_2: \pi^* K_C^{-m}\xrightarrow{\lambda} \O_Y \to A_Y(m), \quad z \mapsto \lambda(z)\cdot  1_{A_Y(m)} \]
\end{itemize}
Then we define the two-sided ideal $I$ in $A_Y$ generated by the difference of $\nu_1$ and $\nu_2$:
\begin{align*}
    I:=\left\langle  Im(\nu_1-\nu_2 )\right\rangle  &=\left\langle \pi^*(\Delta (s) )- \lambda(\pi^*(s))\cdot 1_{A_Y(m)} |s\in K_C^{-m}\right\rangle\\
    &= \left\langle \pi^*(\iota_i(s)) - \lambda( \pi^*(s))\cdot e_{Y,i}|i\in \zm, s\in K_C^{-m}\right\rangle
\end{align*}
where the last equality is obtained by the decomposition of the unit $1_{A_Y(m)} = \sum_{i\in \zm}e_{Y,i}$.
We define the sheaf of algebras on $Y$: 
\[ \wAm := A_Y(m)/I.\]
Hence, the idea of the ideal $I$ is to identify the actions corresponding to loops and scalar multiplication by the tautological section $\lambda$.

\begin{prop}\label{prop:finiterank}
    \leavevmode
\begin{enumerate}
    \item As an $\O_Y$-module, $\wAm$ is locally free of finite rank of $m^2$.

    \item Let $Y_0\subset Y$ denote the zero section of $K_C^m$ over $C$. For any point $y\in Y\setminus Y_0$, the fiber $\wAm|_y$ is isomorphic to the matrix algebra $\operatorname{Mat}_{m\times m}(\C)$.
\end{enumerate}
\end{prop}
\begin{proof}
 (1) The $\O_Y$-algebra $A_Y(m)$ is generated by sections corresponding to all paths. Any path $p$ starting from $v_j$ and ending at $v_i$ can be factored as $p = c_i^\ell \cdot p_{ij}$, where $p_{ij}$ is the unique path of length $< m$ from $v_j$ to $v_i$. Recall that taking the quotient of $A_Y(m)$ by $I$ means that we are identifying the action of the loop $c_i$ with multiplication by the tautological section $\lambda$. Therefore, in the quotient $\wAm=A_Y(m)/I$, the loop $c_i^{\ell}$ just acts as scalar multiplication. It follows that $\wAm$ is generated as an $\O_Y$-module by the sections corresponding to paths of length $<m$. Since the number of paths $p$ of length $<m$ is $m^2$ and all the component $M_{p}$ in $A_Y(m)$ is locally free of rank 1, the quotient $A_Y(m)/I$ is locally free of rank $m^2$.

 (2) Let $y\in Y\setminus Y_0$. By the Wedderburn-Artin theorem, it suffices to check that $\wAm|_{y}$ is a simple $\C$-algebra. Let $J$ be a non-zero two-sided ideal of $\wAm|_{y}$. For simplicity, we will denote by $e_i$ the restriction of the section $e_i$ of $\wAm|_y$ as well. Take any non-zero element $x\in J$. Write $x= \sum_{i,j\in Q_0} x_{i\leftarrow j}$ according to the direct sum decomposition \ref{path decomposition} (which holds for the quotient $\wAm$). Since $x \neq 0$, there must exist a nonzero component $x_{i\leftarrow j} = e_ixe_j$ which also lies in $J$ because $J$ is a two-sided ideal. 
 
 Take a nonzero element $u_{j\leftarrow i}\in \wAm|_y$ corresponding to the path from $v_i$ to $v_j$ (which exists because the component of $\wAm|_y$ for any path of length $<m$ is a 1-dimensional vector space). The product $u_{j\leftarrow i}\cdot x_{i\leftarrow j} $ lies in $e_j\wAm|_ye_j\subset J$ which correpsonds to the loop starting and ending at $v_j$. By the definition of the ideal $I$ of relations, the action corresponding to a loop evaluates to the tautological section $\lambda(y)$ which is non-zero because $y\not \in Y_0$. Thus, $u_{j\leftarrow i}\cdot x_{i\leftarrow j}= b_je_j$ for some non-zero scalar $b_j$ and it follows that $e_j\in J$. 
 
 Once we get $e_j\in J$, we can show $e_k\in J$ for all vertices $v_k$. By choosing a nonzero element $u_{k\leftarrow j}$ and $u_{j\leftarrow k}$ corresponding to the path from $v_j$ to $v_k$ and $v_k$ to $v_j$ respectively. The product $u_{k\leftarrow j}\cdot e_j \cdot u_{j\leftarrow k}\in J$ equals to $b_k e_k$ for some non-zero $b_k$, which implies that $e_k\in J$. Hence, we obtain all the idempotents $e_k$ and the identity element $1 = \sum_{k\in \zm} e_k$. It follows that $J = \wAm|_y$.

\end{proof}

\begin{rem}
We can view the pair $(Y, \wAm)$ as a (midly) noncommutative surface in the sense that the category $\Mod_Y\!\text{-}\wAm$ of quasicoherent right $\wAm$-modules on $Y$ is interpreted as the category $\QCoh(Y, \wAm)$ of coherent sheaves on $(Y, \wAm)$. Note that while the original sheaf of path algebras $A(m)$ is not a coherent $\O_C$-module (being of infinite rank), its central reduction $\wAm$ on $Y$ has finite rank by Proposition \ref{prop:finiterank}. Consequently, the category $\QCoh(Y, \wAm)$ and its derived category are more well-behaved. We refer the readers to \cite[Section 2]{kuznetsov-quadric}\cite[Appendix D]{kuznetsov-hyperplane} for a detailed account of noncommutative algebraic geometry from this point of view. 
\end{rem}

\begin{prop}\label{prop:pushforward}
We have an isomorphism of $\O_C$-algebras:
\[\pi_*\wAm \cong A(m) \]
\end{prop}
\begin{proof}
    Since $\pi$ is affine, we have $\pi_*\wAm\cong \pi_*A_Y(m)/\pi_*I$. By the projection formula, we have a natural isomorphism: 
    \[\pi_*A_Y(m) \cong  \pi_*\O_Y\otimes_{\O_C} A(m) \]
    Recall that $Y= \underline{\operatorname{Spec}}(S^\bullet (K_C^{-m}))$. This implies that $\pi_*\O_Y = S^\bullet K_C^{-m}$ which is isomorphic to the central subalgebra $Z(A(m))$ by Proposition~\ref{central2}. 

    Let $m_C:\pi_*\O_Y\otimes_{\O_C}A(m)\to A(m)$ be the natural multiplication map. Since $\pi_*\O_Y$ maps into the center of $A(m)$, this multiplication is compatible with the algebra structure and $m_C$ is a well-defined homomorphism of $\O_C$-algebras. Furthermore, $m_C$ is clearly surjective since $ \pi_*\O_Y$ contains the identity element. It remains to show that the kernel of $m_C$ is the ideal $\pi_*I$.

    We can verify this locally over an affine open subset $U= \Spec(R)\subset C$ over which the line bundle $K_C^{-m}$ is trivial. Let $A_R:=A(m)(U)$. Over $U$, the central subalgebra $Z(A(m))\cong S^\bullet K_C^{-m}$ restricts to $Z_R\cong R[t]$ where $t$ corresponds to a trivializing section of $K_C^{-m}(U)$. The restriction of $\pi_*A_Y(m)$ to $U$ is simply the $R$-algebra $R[t]\otimes_R A_R$. Locally, the multiplication map $m_C$ becomes 
    \[m_R:R[t]\otimes _RA_R\to A_R \]
    defined by $1\otimes a\mapsto a$ and $t\otimes 1\mapsto \Delta(t)$ where $\Delta$ is the diagonal map defined in \ref{eq:diagonal}. 

    By definition, the ideal $\pi_*I$ is generated by elements of the form $\pi^*(\Delta(s)) -\lambda (\pi^*(s))\cdot 1_{A_Y(m)}$. On the local chart $U$, since any section $s|_U$ is represented as $r\cdot t$ for some $r\in R$. So, the ideal $\pi_*I(U)$ is determined by the relation on $t$, which is the element $1\otimes \Delta(t)-t\otimes 1$. By standard properties of tensor products of algebras, the kernel of the $m_R$ is generated exactly by the two-sided ideal $\langle 1\otimes \Delta(t)- t\otimes 1\rangle$. Therefore, $\ker(m_R)=\pi_*I(U)$. Since this holds for all affine open subset, we conclude that $\ker(m_C)=\pi_*I$. 
    
\end{proof}

\begin{cor}\label{cor:equiv2}
There is an equivalence between $\Mod_Y\!\text{-} \wAm$ and $\Mod_C\!\text{-}A(m)$. 
\end{cor}
\begin{proof} Since $\pi:Y\to C$ is an affine morphism, the pushforward functor $\pi_*: \Mod_Y\!\text{-} \wAm \to \Mod_C\!\text{-}\pi_*\wAm$ is an equivalence of categories, just as in the commutative analogue \cite[\href{https://stacks.math.columbia.edu/tag/01SB}{Tag 01SB}]{stacks-project}. Then the result follows by Proposition \ref{prop:pushforward}.  \end{proof}

\begin{prop}\label{prop:equiv3}
There is an equivalence between the category $\mod_Y\!\text{-}\wAm$ of coherent right $\wAm$-modules on $Y$ and the category of $M_Y$-twisted $Q(m)$-quiver sheaves $(\{F_i\}_{i\in \zm}, \{\psi_i\}_{i\in \zm} )$ with the additional condition that for each $i\in \zm$, the composition of arrows around the loop (suppressing the twist by the identity on $\pi^*K_C^{-1}$) yields:
\[\psi_{i-1}\circ \cdots \circ \psi_i = \operatorname{Id}_{F_i}\otimes_{\O_Y} \lambda  \in Hom_{\O_Y}(F_i\otimes_{\O_Y} \pi^*K_C^{-m}, F_i )\]
where the subcript of $\psi_i$ is understood modulo $m$.
\end{prop}

\begin{proof}
    The category $\mod_Y\!\text{-}\wAm$ is a full subcategory of $\mod_Y\!\text{-}A_Y(m)$ consisting of $A_Y(m)$-modules that annihilates the ideal $I$. By Proposition \ref{prop:equiv1}, an object $\m F\in \mod_Y\!\text{-}A_Y(m)$ corresponds to a $M_Y$-twisted $Q(m)$-quiver sheaf $(\{F_i\}_{i\in \zm}, \{\psi_i\}_{i\in \zm})$. It remains to determine what condition is imposed on this quiver sheaf by requiring it to annihilate $I$. 

    Recall that $I=Im(\nu_1-\nu_2)$. 
    The map $\nu_1: \pi^*K_C^{-m}\to A_Y(m)$ corresponding to loops in the path algebra induces
    \begin{align*}
    \m  F\otimes_{\O_Y} \pi^*K_C^{-m} &\to \m F\otimes_{\O_Y} A_Y(m) \to \m F
    \end{align*}
    By restricting this map to the component $F_i$, we get
    \[\psi_{i-1}\circ \cdots \circ \psi_i : F_i\otimes \pi^*K_C^{-m} \to F_i\]
    On the other hand, the map $\nu_2: \pi^*K_C^{-m}\to A_Y(m)$ is defined by multiplication by the tautological section $\lambda\in H^0(Y, \pi^*K_C^m)$. 
    Its induced action on each component $F_i$ maps a section $f\otimes z\in F_i\otimes_{\O_Y}\pi^*K_C^{-m}$ to $f\cdot \lambda(z)\in F_i$, that is, 
    \[\operatorname{Id}_{F_i}\otimes _{\O_Y}\lambda: F_i\otimes \pi^*K_C^{-m} \to F_i.\]
    Therefore, in order for $\m F$ to annihilates the ideal $I$, each component $F_i$ must satisfy
    \[\psi_{i-1}\circ \cdots \circ \psi_i = \operatorname{Id}_{F_i}\otimes_{\O_Y} \lambda  ,\quad \textrm{for } i\in \zm\]
    as desired.  
\end{proof}

\begin{theorem}\label{thm:equiv4}
The following four categories are equivalent:
\begin{enumerate}
    \item The category of $M$-twisted $Q(m)$-quiver sheaves $(\{E_i\}_{i\in\zm}, \{\phi_i\}_{i\in\zm})$ on $C$.
    
    \item The category of coherent right $A(m)$-modules on $C$.
    
    \item The category of coherent right $\wAm$-modules on $Y$ whose support is finite over $C$.
    
    \item The category of $M_Y$-twisted $Q(m)$-quiver sheaves $(\{F_i\}_{i\in\zm}, \{\psi_i\}_{i\in\zm})$ on $Y$ such that:
    \begin{itemize}
        \item[(a)] The sheaf $F_i$ has support finite over $C$ for all $i\in \zm$. 
        \item[(b)] The relation $\psi_{i-1}\circ \cdots \circ \psi_i = \operatorname{Id}_{F_i}\otimes_{\O_Y} \lambda  $ holds in $\operatorname{Hom}_{\O_Y}(F_i\otimes_{\O_Y} \pi^*K_C^{-m}, F_i )$ for all $i\in \mathbb{Z}/m\mathbb{Z}$.
    \end{itemize}
\end{enumerate}
\end{theorem}

\begin{proof}
The equivalence (1) $\Leftrightarrow$ (2) is Proposition \ref{prop:equiv1}. The equivalence (2)$ \Leftrightarrow$ (3) follows from the Corollary \ref{cor:equiv2}: under the equivalence $\pi_*$, the $\mathcal{O}_C$-coherence condition for an $A(m)$-module translates precisely to the finite support condition for an $\wAm$-module on $Y$. The equivalence (3) $\Leftrightarrow$ (4) is a restriction of Proposition~\ref{prop:equiv3} to $\mathcal{O}_Y$-coherent sheaves with support finite over $C$.
\end{proof}

Let $(\mathcal{E}, \Phi) =(\{E_i\}_{i\in \zm}, \{\phi_i\}_{i\in \zm})$ be a cyclic Higgs bundle. For each $i \in \mathbb{Z}/m\mathbb{Z}$, the composite map for the loop at $v_i$, 
$$ \Phi_i := \phi_{i-1} \circ \dots \circ \phi_i : E_i \otimes K_C^{-m} \to E_i $$
is an $\mathcal{O}_C$-linear map which defines a spectral curve $\tC_i =\{\det(\lambda Id - \pi^*\Phi_i)=0\}  \subset Y$ where $\lambda\in H^0(Y, \pi^*K_C^{-m})$ is the tautological section. The restriction $\pi|_{\tC_i}:\tC_i\to C$ is a $p_i$-sheeted cover of $C$, where $p_i = \operatorname{rank}(E_i)$. The space of spectral curves is parameterized by the coefficients of the characteristic polynomials which lie in the vector space
\[ B(p_i)= \bigoplus_{\mu=1}^{p_i} H^0(C, K_C^{m\mu} )\] 
By a slight abuse of notation, we will write $\tC_i \in B(p_i)$ to denote the spectral curve $\tC_i$ corresponding to that point in the parameter space $B(p_i)$. For a given dimension vector $\bm{p}=(p_0,\dots, p_{m-1})$, let $B(\bm{p})= B(p_0)\times \cdots \times B(p_{m-1})$. 
\begin{cor}\label{cor:equiv5}
    Fix a dimension vector $\bm{p}=(p_0, \dots, p_{m-1})$ and a degree vector $\bm{d}=(d_0, \dots, d_{m-1})$.
Let $(\tC_0, \dots,\tC_{m-1})\in B(\bm{p})$ be a fixed collection of spectral curves in $Y$. The following two categories are equivalent:
\begin{enumerate}
    \item (Higgs data) The category of cyclic Higgs bundles $(\{E_i\}_{i\in \zm}, \{\phi_i\}_{i\in \zm})$ on $C$ with dimension vector $\bm{p}$ and degree vector $\bm{d}$, such that the spectral curve of  $\Phi_i$ is $\tC_i$ for all $i\in \zm$.
    
    \item (Spectral data) The category of $M_Y$-twisted $Q(m)$-quiver sheaves $(\{F_i\}_{i\in\zm}, \{\psi_i\}_{i\in\zm})$ on $Y$ satisfying:
    \begin{itemize}
        \item[(a)] For each $i \in \mathbb{Z}/m\mathbb{Z}$, $F_i$ is a pure dimension one sheaf on $Y$ that is scheme-theoretically supported on $\tC_i$ (i.e. an $\O_{\tC_i}$-module) with Euler characteristic $\chi(F_i)= d_i + p_i(1-g)$. 
        \item[(b)] The relation $\psi_{i-1}\circ\cdots \circ \psi_{i} = \operatorname{Id}_{F_i}\otimes_{\O_Y} \lambda $ holds in $\operatorname{Hom}_{\O_Y}(F_i\otimes_{\O_Y} \pi^*K_C^{-m}, F_i )$ for all $i\in \mathbb{Z}/m\mathbb{Z}$. 
    \end{itemize}
\end{enumerate}
\end{cor}
\begin{proof}
    The correspondence between the cyclic Higgs bundles $(\{E_i\}_{i\in \zm}, \{\phi_i\}_{i\in \zm})$ and quiver sheaves $(\{F_i\}_{i\in\zm}, \{\psi_i\}_{i\in\zm})$ is the equivalence (1) $\Leftrightarrow$ (4) in Theorem~\ref{thm:equiv4}. In particular, under this equivalence, we have $E_i =\pi_* F_i$ and the Higgs field $\Phi_i$ is recovered by pushing forward $\psi_{i-1}\circ\cdots \circ \psi_{i} =  \operatorname{Id}_{F_i}\otimes_{\O_Y}\lambda$ to $C$. For each $i$, this reduces precisely to the classical spectral correspondence for $GL(p_i, \C)$-Higgs bundles (\cite{bnr, simpsonmoduli}). Thus, the condition that $E_i$ is locally free of rank $p_i$ and degree $d_i$ with spectral curve $\tC_i$ translates directly to $F_i$ being pure dimension one with support $\tC_i$ and Euler characteristic $d_i+p_i(1-g)$.  
    
\end{proof}

\begin{rem}[Constraints on spectral curves]\label{rem:constraints}
Suppose $(\tC_0, \dots,\tC_{m-1})\in B(\bm{p})$ is a fixed collection of spectral curves in $Y$ that comes from a cyclic Higgs bundle. Then the equivalence in Corollary~\ref{cor:equiv5} imposes a constraint on the spectral curves. Indeed, the relations 
\[(\psi_{i-1}\circ\cdots \circ\psi_{i+1})\circ \psi_{i} = \operatorname{Id}_{F_i}\otimes_{\O_Y} \lambda,\quad \psi_{i}\circ (\psi_{i-1}\circ\cdots \circ \psi_{i+1})= \operatorname{Id}_{F_{i+1}}\otimes_{\O_Y} \lambda \]
implies that $\psi_{i}:F_i\otimes \pi^*K_C^{-1}|_{Y\setminus Y_0}\to F_{i+1}|_{Y\setminus Y_0} $ is an isomorphism since $\operatorname{Id}_{F_i}\otimes_{\O_Y} \lambda $ is an isomorphism on $Y\setminus Y_0$. Moreover, taking the tensor product with the line bundle $\pi^{*}K_C^{-1}$ does not change the support of $F_i$, so $F_i$ and $F_{i+1}$ share the same support away from $Y_0$. In particular, there exists a common curve $\tC\subset Y$ such that $(\tC_0, \dots,\tC_{m-1})\in B(\bm{p})$ is of the form 
\begin{equation}\label{eq:spectral curve}
    (\tC + q_0 Y_0 , \dots, \tC + q_{m-1} Y_0) ,\quad \textrm{ where }q_0,\dots, q_{m-1} \in \Z_{\geq 0}.
\end{equation}

\end{rem}

\section{Examples}\label{sec:examples}
\subsection{$GL(r, \C)$-Higgs bundles}
For $m=1$, $Z(A(1))\cong A(1)$ and $\wA(1) \cong \O_Y$. In this case the noncommutative surface $(Y, \wA(1))= (Y, \O_Y)$ is just the standard commutative ruled surface $Y$. The correspondence (Corollary~\ref{cor:equiv5}) reduces to the classical result of BNR \cite{bnr} and Simpson \cite{simpsonmoduli}. 

\subsection{$U(p_0,p_1)$-Higgs bundles and even Clifford algebras}\label{sec:clifford}
    Consider $m=2$ and $\bm{p}= (p_0,p_1)$. In this case, we will show that the noncommutative surface $(Y, \wA(2))$ can be identified with another noncommutative surface derived from a conic fibration over $Y$ with ramification locus $Y_0$. 

    Let us first recall the definition of conic fibration and its associated sheaf of even Clifford algebras. Let $S$ be a smooth variety. 
    Fix a rank $3$ vector bundle $F$ on $S$ and a  quadratic form on $F$ which is an embedding of a line bundle $q: L \to S^2F^{\vee}$. Let $f':\P(F):= \operatorname{Proj}(S^\bullet F^{\vee})\to S$ be the projection map. Since the quadratic form $q$ is equivalent to a section 
    \[s_q\in H^0(S, S^2 F^{\vee}\otimes L^{-1}) = H^0(\P(F), \O_{\P(F)/S}(2)\otimes f'^*L^{-1}),\]
    the zero locus $\m X$ of $s_q$ defines a conic fibration $f=f'|_{\m X}: \m X\to S$ whose fibers are plane conics. The discriminant locus of $f$ is the locus where the rank of the quadratic form drops. 
    
    It is known by the work of Kuznetsov \cite{kuznetsov-quadric} that the bounded derived category $D^b(\m X)$ of the conic fibration $\m X$ over $S$ admits a semiorthogonal decomposition 
    \[ D^b(\m X) = \langle D^b(S, \mathcal{C}\ell_0(q)), f^*D^b(S) \rangle\]
    where the first component is the bounded derived category of right modules over the sheaf of even Clifford algebra $\mathcal{C}\ell_0(q)$. Just as a quadratic form on a vector space defines the Clifford algebra which decomposes into the even and odd parts, the sheaf of algebras $\mathcal{C}\ell_0(q)$ is obtained as a relative version of this construction: it is the sheaf of $\O_S$-algebras whose fiber at any point $x\in S$ is the even part of the Clifford algebra defined by the quadratic form $q_x$ on the fiber $F_x$ (see the original paper of Kuznetsov \cite{kuznetsov-quadric} for more details). 

    The following proposition identifies the sheaf of algebras $\wA(2)$ as the sheaf of even Clifford algebras of a conic fibration over $Y$. 

    \begin{prop}\label{prop:clifford}
        There exists a rank 3 vector bundle $F$ over $Y$ and a quadratic form $q$ on $F$ with discriminant locus $Y_0$ such that the associated sheaf of even Clifford algebras $\mathcal{C}\ell_0(q)$ is isomorphic to $\wA(2)$. 
    \end{prop}
    \begin{proof}
        It is possible to explicitly write down the quadratic form and construct an isomorphism between the two sheaves of algebras. Instead, we will apply a useful criterion of Kuznetsov which essentially reduces to checking that $\wA(2)$ is an even Clifford algebra over each point $y\in Y$. The idea is to show that $\wA(2)$ is a pointwise Clifford algebra in the sense of Kuznetsov  which implies the existence of the desired quadratic form by \cite[Proposition 2.7]{kuznetsov2025spinormodificationsconicbundles}. We will focus on checking the required condition of pointwise Clifford algebra and refer the readers to \cite[Definition 2.5]{kuznetsov2025spinormodificationsconicbundles} for a precise definition. 

        By the analysis in Proposition~\ref{prop:finiterank}, $\wA(2)$ admits a direct sum decomposition according to the length of paths. More specifically, we have $\wA(2) \cong \O_Ye_{0}\oplus \O_Y e_{1}\oplus M_{0\leftarrow 1} \oplus M_{1\leftarrow 0}$ as $\O_Y$-modules where $M_{1\leftarrow 0}$ and $M_{0\leftarrow 1}$ are both isomorphic to $\pi^*K_C^{-1}$. We can define an $\O_Y$-linear trace map via:
        \[ \operatorname{Tr}: \wA(2)\to \O_Y, \quad \a= \a_0 e_0+ \a_1e_1+ \a_{0\leftarrow 1} + \a_{1\leftarrow 0} \mapsto \a_0+\a_1 \]
        By the definition of $\wA(2)$, the sections corresponding to loops are identified with $\O_Y$-scalar multiplication by the tautological section $\lambda$. So, for any $\alpha_{0\leftarrow 1}\in M_{0\leftarrow 1}$ and $\alpha'_{1\leftarrow 0}\in M_{1\leftarrow 0}$, their products yield $\alpha_{0\leftarrow 1}\alpha'_{1\leftarrow 0} = \lambda e_0$ and  $\alpha'_{1\leftarrow 0}\alpha_{0\leftarrow 1} = \lambda e_1$. The trace of their commutator 
        \[\operatorname{Tr}([\alpha_{0\leftarrow 1},\alpha'_{1\leftarrow 0}]) = \operatorname{Tr}(\lambda e_0 - \lambda e_1) = \lambda - \lambda =0 .\]
        This implies that the trace map also vanishes on $[\wA(2), \wA(2)]$. 
        
        Moreover, $\operatorname{Tr}$ is clearly surjective and the section $\frac{1}{2}(e_0+e_1): O_Y\to \wA(2)$ provides a splitting. Hence, we have a direct sum decomposition
        \begin{equation}\label{eq:decomposition}
            \wA(2) = \O_Y(e_0+e_1)\oplus \wA(2)^0
        \end{equation}
        where $\wA(2)^0= \ker (\operatorname{Tr})$ contains the commutator subsheaf $[\wA(2), \wA(2)]$. This verifies the first half of the definition of a pointwise Clifford algebra. 

        For any point $y\in Y_0$, the fiber $\wA(2)|_y$ is a path algebra of the quiver $Q(2)$ with relations. It is observed in \cite[Proposition 4.1]{lee-clifford} that $\wA(2)|_y$ is isomorphic to the even Clifford algebra associated to a degenerate quadratic form. Moreover, the explicit isomorphism provided in \cite[Eq. (19)]{lee-clifford} identifies the direct sum decomposition \ref{eq:decomposition} with the natural decomposition of the even Clifford algebra. 

        For any point $y\in Y\setminus Y_0$, the fiber $\wA(2)|_y$ is a matrix algebra by Proposition~\ref{prop:finiterank}(2) which is the even Clifford algebra associated to a nondegenerate quadratic form (e.g. see \cite[Remark 2.4]{kuznetsov2025spinormodificationsconicbundles}). Moreover, the trace map on a matrix algebra is unique, and so is the direct sum decomposition \ref{eq:decomposition}. Hence, the isomorphism between $\wA(2)|_y$ and $\operatorname{Mat}_{2\times 2}(\C)$ identifies the the decompositions. 

        Therefore, $\wA(2)$ is a pointwise Clifford algebra and the result follows from applying \cite[Proposition 2.7]{kuznetsov2025spinormodificationsconicbundles}. 
        Moreover, by its construction, the vector bundle $F$ is given by $(\wA(2)^0)^{\vee}$. 
    \end{proof}

        From the perspective of even Clifford algebras, there is a further central reduction of the sheaf of algebras $\wA(2)\cong \mathcal{C}\ell_0(q)$ into a sheaf of Azumaya algebras $\widehat{\mathcal{C}\ell_0(q)}$ on the $2$-nd root stack $\widehat{Y}$ of $Y$ along $Y_0$ \cite[Section 3.6]{kuznetsov-quadric} (see \cite[Section 3.1-3.2]{lee-clifford} for a summary). In particular, we have 
        \begin{equation}\label{eq:azumaya}
            \operatorname{Coh}(\widehat{Y}, \widehat{\mathcal{C}\ell_0(q)})\cong \operatorname{Coh}(Y, \mathcal{C}\ell_0(q))
        \end{equation}
        Note that there is a natural morphism $S^\bullet K_C^{-2}\to S^{\bullet} K_C^{-1}$ of $\O_C$-algebras which induces a (2-sheeted) cyclic cover $\eta: \tY = \underline{\operatorname{Spec}}(S^{\bullet} K_C^{-1})\to\underline{\operatorname{Spec}}(S^{\bullet} K_C^{-2})= Y$ with a $\Z/2\Z$-action and ramified at $Y_0$. So, $\widehat{Y}$ is the quotient stack $\left[\tY/ (\Z/2\Z)\right]$. If we restrict the equivalence \ref{eq:azumaya} to a smooth curve $\tC\subset Y$ which intersects $Y_0$ transversally (which happens for the case in Section~\ref{sec:equal dim}), then the sheaf of Azumaya algebras splits \cite[Corollary 3.16]{kuznetsov-quadric} and the equivalence~\ref{eq:azumaya} becomes 
        \[     \operatorname{Coh}\left(\left[\Sigma/ (\Z/2\Z)\right]\right)\cong \operatorname{Coh}(\tC, \mathcal{C}\ell_0(q)|_{\tC}) \]
        where $\Sigma\subset \tilde{Y}$ is the double cover of $\tC$. When $\tC\subset Y$ is a smooth spectral curve for cyclic Higgs bundles of dimension vector $(p,p)$, the double cover $\Sigma\subset \widetilde{Y}$ is exactly the spectral curve of the associated $GL(2p,\C)$-Higgs bundle. Hence, the equivalence implies that cyclic Higgs bundles of dimension vector $(p,p)$ with spectral curve $\tC$ corresponds to $\Z/2\Z$-equivariant line bundles on $\Sigma$. This description is consistent with the characterization of $U(p,p)$-Higgs bundles as fixed points of an involution (see \cite[Proposition 3]{schaposnik14}).

    \begin{rem}
        On a related note, we remark that conic fibrations also arise in the context of $PGL(r,\C)$-Hitchin systems. By the work of Diaconescu--Donagi--Pantev \cite{DDP} (see also \cite[Section 8.8]{kontsevich-soibelman} for the affine conic fibration description), 
        there is a family of (Calabi-Yau) affine conic fibrations whose intermediate Jacobians are isomorphic to the Hitchin fibers of $PGL(r,\C)$-Higgs bundles. In this case, the discriminant locus of the affine conic fibration is the corresponding spectral curve.

    \end{rem}

\subsection{Spectral data for cyclic Higgs bundles of dimension vector $(p,\dots, p)$}\label{sec:equal dim}
Consider $\bm{p} = (p,\dots, p)$ and fix $(\tC_0,\dots, \tC_{m-1})\in B(\bm{p})$. By Remark~\ref{rem:constraints}, the degree of $\tC_i$ over $C$ is $p$ and so $q_0= \dots = q_{m-1}=0$ in \ref{eq:spectral curve}. This means that the associated spectral curves $\tC_0= \dots= \tC_{m-1}=\tC$  coincide. 

Suppose the common spectral curve $\tC$ is smooth and not equal to $Y_0$. The spectral data $(\{F_i\}_{i\in\zm}, \{\psi_i\}_{i\in\zm})$ in Corollary~\ref{cor:equiv5} can be further simplified as follows. For $i\in \zm$, the pure dimension one sheaf $F_i$ on $Y$ is a line bundle $L_i$ on $\tC$. Note that
\[\psi_i\in \operatorname{Hom}_{\tC}(L_i\otimes\pi^*K_C^{-1}, L_{i+1})\cong H^0(\tC,L_i^{-1}\otimes L_{i+1}\otimes  \pi^*K_C).\]
Since $\psi_i$ is an isomorphism away from $Y_0$, we see that $\psi_i$ is a non-zero section of $L_i^{-1}\otimes L_{i+1}\otimes  \pi^*K_C$ and determines an effective divisor $D_i = \operatorname{div}(\psi_i)$. Conversely, given $L_i$ and an effective divisor $D_i$, we can recover $L_{i+1}$ as $\O(D_i)\otimes L_i\otimes \pi^*K_C^{-1}$ and $\psi_i$ as the unique section that vanishes on $D_i$. The relation $\psi_{i-1}\circ\cdots \circ \psi_{i} = \operatorname{Id}_{F_i}\otimes_{\O_Y} \lambda $ translates into the relation that 
\begin{equation}\label{eq:relation}
    D_0+\dots +D_{m-1}= \operatorname{div}(\lambda|_{\tC})
\end{equation}
By starting with a line bundle $L_0$ on $\tC$ and running the argument above, we conclude that the spectral data $(\{F_i\}_{i\in\zm}, \{\psi_i\}_{i\in\zm})$ in Corollary~\ref{cor:equiv5} is equivalent to the data of $(L_0, D_0,\dots, D_{m-1})$ where $L_0$ is a line bundle on $\tC$ and each $D_i$ is an effective divisor on $\tC$ satisfying the relation \ref{eq:relation}.

For $\bm{p}= (p,p)$, the description above recovers the spectral data for $U(p,p)$-Higgs bundles studied in the work of Schaposnik \cite{schaposnik14}.

\section{Future Directions}\label{sec:future}
\subsection{Spectral data for cyclic Higgs bundles of arbitrary dimension vector $(p_0,\dots,p_{m-1}$) }

The primary reason that the spectral data in the case of $(p,\dots, p)$ admits a simplification to $(L_0, D_0, \dots, D_{m-1})$ is that the associated spectral curves $\tC_0=\dots = \tC_{m-1}$ coincide. However, as soon as one of the components in the dimension vector is different from the others, at least one of the spectral curves will be reducible. 

For example, consider $\bm{p} = (p+1,\dots, p)$ and fix $(\tC_0,\dots, \tC_{m-1})\in B(\bm{p})$. As noted in Remark~\ref{rem:constraints}, we must have $q_0=1$ and $q_i=0$ for $i\neq 0$ in \ref{eq:spectral curve}. This implies that the associated spectral curves are given by $\tC_0 = \tC \cup Y_0$ and $\tC_1= \dots= \tC_{m-1}=\tC$, where $\tC$ is a degree $p$ cover over $C$. Then it is possible to reduce the spectral data in Corollary~\ref{cor:equiv5} in terms of spectral data on $\tC$ for dimension vector $(p,\dots, p)$, a line bundle on $Y_0\cong C$ and a gluing data (such as \cite[Corollary 5.10]{peónnieto2015cameraldatasup1phiggsbundles}). However, the complexity of this kind of analysis grows as the difference between the components in the dimension vector increases. 

\subsection{Moduli spaces of cyclic Higgs bundles}
As mentioned in the introduction, the moduli space of cyclic Higgs bundles can be viewed as the fixed point locus of an $\zm$-action on the moduli space of $GL(r,\C)$-Higgs bundle. By the spectral corresondence established in Corollary~\ref{cor:equiv5} and Proposition~\ref{prop:equiv3}, we can interpret this as the moduli space of pure dimension one sheaves on the noncommutative surface $(Y,\wAm)$. One can argue the existence of its coarse moduli space by defining the appropriate stability conditions and show that the sheaf $\wAm$ of algebras is a sheaf of rings of differential operators in the sense of Simpson \cite{Simpson94moduli1}. Then the general theory of Simpson guarantees the existence of moduli space of semistable $\wAm$-modules on $Y$. 

A potential application of this reinterpretation is an algebraic approach to study their connected components. Recall from Section~\ref{sec:clifford} that $\wAm$ is isomorphic to the sheaf of even Clifford algebras of a conic fibration on $Y$. It is proven in the work of Lahoz--Macr\`{\i}--Stellari \cite[Theorem 2.12]{LahozMacriStellari} that the moduli space of modules over even Clifford algebras which arises from a cubic threefold is irreducible and hence connected. Their proof is a variant of the classical argument of Mukai that originally establishes the connectedness of the moduli space of sheaves on a K3 surface. Therefore, we expect a variant of Mukai's argument can be used to prove the connectedness of the moduli space of $\wAm$-modules on $Y$ (after imposing a suitable numerical invariant such as the Toledo invariant). This would provide an algebraic approach that complements the Morse-theoretic approach of Bradlow--Garc\'{i}a-Prada--Gothen \cite{bradlow-gp-gothen} to study the connected components of the moduli space of $U(p,q)$-Higgs bundles. 

\subsection{Generalized Hitchin fibration for quiver bundles}
From the quiver theoretic viewpoint, cyclic Higgs bundles represent the quiver bundles of a specific quiver $Q(m)$. It would be interesting to study the spectral correspondence for quiver bundles of other types of quivers. A quick observation is that the reason that the case of the quiver $Q(m)$ works well is that the center of the associated path algebra is exactly $S^\bullet K_C^{-m}$ (Proposition~\ref{central2}). It is known from the general theory of quiver that a quiver without oriented loop has trivial center. So, the natural first step to analyze  the case for other quivers will be to look at quivers with oriented loops and then study the center subalgebra of their path algebras. 

From the moduli viewpoint, the spectral correspondence for quiver bundles is equivalent to the study of the fibers of the corresponding generalized Hitchin fibration (originally defined by Schmitt in \cite{schmitt}). To give a brief description of the generalized Hitchin fibration, it is more convenient to use the language of stacks. 

Let $Q=(Q_0, Q_1)$ be a general quiver with head and tail maps $h,t:Q_1\to Q_0$. For a fixed dimension vector $\bm{p}=(p_i)_{i\in Q_0}$, we associate to each vertex $i\in Q_0$ a complex vector space $V_i$ of dimension $p_i$. The space of representations of $Q$ with dimension vector $\bm{p}$ is the affine space 
\[ \operatorname{Rep}(Q,\bm{p})= \bigoplus_{a\in Q_1} \operatorname{Hom}(V_{t(a)}, V_{h(a)}). \]
The change of basis group is defined as the product
\[ G_{\bm{p}}:= \prod_{i\in Q_0}GL(V_i)\]
which acts naturally on $\operatorname{Rep}(Q,\bm{p})$ by conjugation. Given this action, there is the associated morphism from the quotient stack to the GIT quotient: 
\[\chi: \left[ \operatorname{Rep}(Q, \bm{p})/G_{\bm{p}} \right] \to \operatorname{Rep}(Q, \bm{p})/\!/G_{\bm{p}} \] 
Just as the original Hitchin fibration for $G$-Higgs bundles is induced by $\left[\mathfrak{g}/G\right] \to \mathfrak{g}/\!/G$, by following the general construction as in \cite{ngo2006fibration,ngo2010lemme}, the morphism $\chi$ induces a generalized Hitchin fibration from the moduli stack of quiver bundles to an affine Hitchin base. 

In fact, Ng\^{o} has recently initiated a general framework to study the generalized Hitchin fibration induced from the action of a reductive group on an affine normal scheme (see the related work \cite{hameister}). The action of $G_{\bm{p}}$ on $\operatorname{Rep}(Q,\bm{p})$ can be seen as an example of this general framework. His approach is closer to the original cameral cover approach that works well over the regular locus. Our emphasis here is on the quiver structure and a description over the whole Hitchin base. It will be interesting to study the relation between the different approaches.

\printbibliography

@article {schaposnik14,
    AUTHOR = {Schaposnik, Laura P.},
     TITLE = {Spectral data for {$U(m,m)$}-{H}iggs bundles},
   JOURNAL = {Int. Math. Res. Not. IMRN},
  FJOURNAL = {International Mathematics Research Notices. IMRN},
      YEAR = {2015},
    NUMBER = {11},
     PAGES = {3486--3498},
      ISSN = {1073-7928,1687-0247},
   MRCLASS = {14J60 (14D20)},
  MRNUMBER = {3373057},
MRREVIEWER = {Sanjay\ Kumar\ Singh},
}

@article{CGP,
      author ={\'Alvarez-C\'onsul, L. and Garc\'ia-Prada, O.} ,
      title = {Hitchin--Kobayashi Correspondence, Quivers,
and Vortices},
number={238},
pages={1--33},
      journal = {Commun. Math. Phys. },
      year = {2003},
DOI={Commun. Math. Phys. 238, 1–33 (2003}
  }

@misc{kuznetsov2025spinormodificationsconicbundles,
      title={Spinor modifications of conic bundles and derived categories of 1-nodal Fano threefolds}, 
      author={Alexander Kuznetsov},
      year={2025},
      eprint={2502.02082},
      archivePrefix={arXiv},
      primaryClass={math.AG},
      url={https://arxiv.org/abs/2502.02082}, 
}

@article {LahozMacriStellari,
    AUTHOR = {Lahoz, Mart\'i{} and Macr\`i, Emanuele and Stellari, Paolo},
     TITLE = {Arithmetically {C}ohen-{M}acaulay bundles on cubic threefolds},
   JOURNAL = {Algebr. Geom.},
  FJOURNAL = {Algebraic Geometry},
    VOLUME = {2},
      YEAR = {2015},
    NUMBER = {2},
     PAGES = {231--269},
      ISSN = {2313-1691,2214-2584},
   MRCLASS = {14F05 (14E05 14J60)},
  MRNUMBER = {3350158},
MRREVIEWER = {Cristian\ V.\ Anghel},
       DOI = {10.14231/AG-2015-011},
       URL = {https://doi.org/10.14231/AG-2015-011},
}

@article {Simpson-middleconvolution,
    AUTHOR = {Simpson, Carlos},
     TITLE = {Katz's middle convolution algorithm},
   JOURNAL = {Pure Appl. Math. Q.},
  FJOURNAL = {Pure and Applied Mathematics Quarterly},
    VOLUME = {5},
      YEAR = {2009},
    NUMBER = {2},
     PAGES = {781--852},
      ISSN = {1558-8599,1558-8602},
   MRCLASS = {14F20 (14D05 20G20 30F30)},
  MRNUMBER = {2508903},
MRREVIEWER = {Vladimir\ P.\ Kostov},
       DOI = {10.4310/PAMQ.2009.v5.n2.a8},
       URL = {https://doi.org/10.4310/PAMQ.2009.v5.n2.a8},
}

@article {kuznetsov-hyperplane,
    AUTHOR = {Kuznetsov, A. G.},
     TITLE = {Hyperplane sections and derived categories},
   JOURNAL = {Izv. Ross. Akad. Nauk Ser. Mat.},
  FJOURNAL = {Izvestiya Rossiiskoi Akademii Nauk. Seriya Matematicheskaya},
    VOLUME = {70},
      YEAR = {2006},
    NUMBER = {3},
     PAGES = {23--128},
      ISSN = {1607-0046,2587-5906},
   MRCLASS = {14F05 (18E30)},
  MRNUMBER = {2238172},
MRREVIEWER = {Andrei\ D.\ Halanay},
       DOI = {10.1070/IM2006v070n03ABEH002318},
       URL = {https://doi.org/10.1070/IM2006v070n03ABEH002318},
}

@article {kuznetsov-quadric,
    AUTHOR = {Kuznetsov, Alexander},
     TITLE = {Derived categories of quadric fibrations and intersections of
              quadrics},
   JOURNAL = {Adv. Math.},
  FJOURNAL = {Advances in Mathematics},
    VOLUME = {218},
      YEAR = {2008},
    NUMBER = {5},
     PAGES = {1340--1369},
      ISSN = {0001-8708,1090-2082},
   MRCLASS = {14F05},
  MRNUMBER = {2419925},
MRREVIEWER = {Adrian\ Langer},
       DOI = {10.1016/j.aim.2008.03.007},
       URL = {https://doi.org/10.1016/j.aim.2008.03.007},
}

@misc{stacks-project,
  author       = {The {Stacks project authors}},
  title        = {The Stacks project},
  howpublished = {\url{https://stacks.math.columbia.edu}},
  year         = {2025},
}

@article {bnr,
    AUTHOR = {Beauville, Arnaud and Narasimhan, M. S. and Ramanan, S.},
     TITLE = {Spectral curves and the generalised theta divisor},
   JOURNAL = {J. Reine Angew. Math.},
  FJOURNAL = {Journal f\"ur die Reine und Angewandte Mathematik. [Crelle's
              Journal]},
    VOLUME = {398},
      YEAR = {1989},
     PAGES = {169--179},
      ISSN = {0075-4102,1435-5345},
   MRCLASS = {14H60 (14D20 14H42)},
  MRNUMBER = {998478},
MRREVIEWER = {P.\ E.\ Newstead},
       DOI = {10.1515/crll.1989.398.169},
       URL = {https://doi.org/10.1515/crll.1989.398.169},
}

@article {simpsonmoduli,
    AUTHOR = {Simpson, Carlos T.},
     TITLE = {Moduli of representations of the fundamental group of a smooth
              projective variety. {II}},
   JOURNAL = {Inst. Hautes \'Etudes Sci. Publ. Math.},
  FJOURNAL = {Institut des Hautes \'Etudes Scientifiques. Publications
              Math\'ematiques},
    NUMBER = {80},
      YEAR = {1994},
     PAGES = {5--79},
      ISSN = {0073-8301,1618-1913},
   MRCLASS = {14D20 (14D22 14F05 14F10)},
  MRNUMBER = {1320603},
MRREVIEWER = {Nitin\ Nitsure},
       URL = {http://www.numdam.org/item?id=PMIHES_1994__80__5_0},
}

@article {lee-clifford,
    AUTHOR = {Lee, Jia Choon},
     TITLE = {Moduli spaces of modules over even {C}lifford algebras and
              {P}rym varieties},
   JOURNAL = {Math. Z.},
  FJOURNAL = {Mathematische Zeitschrift},
    VOLUME = {304},
      YEAR = {2023},
    NUMBER = {3},
     PAGES = {Paper No. 53, 27},
      ISSN = {0025-5874,1432-1823},
   MRCLASS = {14D20 (14F08 14H40)},
  MRNUMBER = {4611792},
MRREVIEWER = {Wolfgang\ Rump},
       DOI = {10.1007/s00209-023-03310-w},
       URL = {https://doi.org/10.1007/s00209-023-03310-w},
}

@article {DDP,
    AUTHOR = {Diaconescu, D. E. and Donagi, R. and Pantev, T.},
     TITLE = {Intermediate {J}acobians and {$ADE$} {H}itchin systems},
   JOURNAL = {Math. Res. Lett.},
  FJOURNAL = {Mathematical Research Letters},
    VOLUME = {14},
      YEAR = {2007},
    NUMBER = {5},
     PAGES = {745--756},
      ISSN = {1073-2780},
   MRCLASS = {14J81 (14J32 14K30 32G81 37J35 81T30)},
  MRNUMBER = {2350120},
MRREVIEWER = {Justin\ Sawon},
       DOI = {10.4310/MRL.2007.v14.n5.a3},
       URL = {https://doi.org/10.4310/MRL.2007.v14.n5.a3},
}

@incollection {kontsevich-soibelman,
    AUTHOR = {Kontsevich, Maxim and Soibelman, Yan},
     TITLE = {Wall-crossing structures in {D}onaldson-{T}homas invariants,
              integrable systems and mirror symmetry},
 BOOKTITLE = {Homological mirror symmetry and tropical geometry},
    SERIES = {Lect. Notes Unione Mat. Ital.},
    VOLUME = {15},
     PAGES = {197--308},
 PUBLISHER = {Springer, Cham},
      YEAR = {2014},
      ISBN = {978-3-319-06513-7; 978-3-319-06514-4},
   MRCLASS = {14N35 (14J33 53D37)},
  MRNUMBER = {3330788},
MRREVIEWER = {Victor\ Przyjalkowski},
       DOI = {10.1007/978-3-319-06514-4\_6},
       URL = {https://doi.org/10.1007/978-3-319-06514-4_6},
}

@article {Li-Mochizuki2,
    AUTHOR = {Li, Qiongling and Mochizuki, Takuro},
     TITLE = {Complete solutions of {T}oda equations and cyclic {H}iggs
              bundles over non-compact surfaces},
   JOURNAL = {Int. Math. Res. Not. IMRN},
  FJOURNAL = {International Mathematics Research Notices. IMRN},
      YEAR = {2025},
    NUMBER = {7},
     PAGES = {Paper No. rnaf081, 38},
      ISSN = {1073-7928,1687-0247},
   MRCLASS = {32W50 (14H60 30F99 37K10)},
  MRNUMBER = {4888699},
       DOI = {10.1093/imrn/rnaf081},
       URL = {https://doi.org/10.1093/imrn/rnaf081},
}

@article {baraglia,
    AUTHOR = {Baraglia, David},
     TITLE = {Cyclic {H}iggs bundles and the affine {T}oda equations},
   JOURNAL = {Geom. Dedicata},
  FJOURNAL = {Geometriae Dedicata},
    VOLUME = {174},
      YEAR = {2015},
     PAGES = {25--42},
      ISSN = {0046-5755,1572-9168},
   MRCLASS = {53C07 (17B80 53C43)},
  MRNUMBER = {3303039},
       DOI = {10.1007/s10711-014-0003-2},
       URL = {https://doi.org/10.1007/s10711-014-0003-2},
}

@article {hitchin-stable,
    AUTHOR = {Hitchin, Nigel},
     TITLE = {Stable bundles and integrable systems},
   JOURNAL = {Duke Math. J.},
  FJOURNAL = {Duke Mathematical Journal},
    VOLUME = {54},
      YEAR = {1987},
    NUMBER = {1},
     PAGES = {91--114},
      ISSN = {0012-7094,1547-7398},
   MRCLASS = {58F07 (14F05 32G13 32L05 32L10)},
  MRNUMBER = {885778},
MRREVIEWER = {G.\ M.\ Khenkin},
       DOI = {10.1215/S0012-7094-87-05408-1},
       URL = {https://doi.org/10.1215/S0012-7094-87-05408-1},
}

@article {bradlow-gp-gothen,
    AUTHOR = {Bradlow, Steven B. and Garc\'ia-Prada, Oscar and Gothen, Peter
              B.},
     TITLE = {Surface group representations and {${\rm U}(p,q)$}-{H}iggs
              bundles},
   JOURNAL = {J. Differential Geom.},
  FJOURNAL = {Journal of Differential Geometry},
    VOLUME = {64},
      YEAR = {2003},
    NUMBER = {1},
     PAGES = {111--170},
      ISSN = {0022-040X,1945-743X},
   MRCLASS = {53D30 (57R19)},
  MRNUMBER = {2015045},
MRREVIEWER = {Ignasi\ Mundet-Riera},
       URL = {http://projecteuclid.org/euclid.jdg/1090426889},
}

@article {gp-peon-nieto,
    AUTHOR = {Garc\'ia-Prada, Oscar and Pe\'on-Nieto, Ana},
     TITLE = {Abelianization of {H}iggs bundles for quasi-split real groups},
   JOURNAL = {Transform. Groups},
  FJOURNAL = {Transformation Groups},
    VOLUME = {28},
      YEAR = {2023},
    NUMBER = {1},
     PAGES = {285--325},
      ISSN = {1083-4362,1531-586X},
   MRCLASS = {14H60 (14D23)},
  MRNUMBER = {4552177},
MRREVIEWER = {Yusuf\ Mustopa},
       DOI = {10.1007/s00031-021-09658-9},
       URL = {https://doi.org/10.1007/s00031-021-09658-9},
}

@incollection {gp-vinberg,
    AUTHOR = {Garc\'ia-Prada, Oscar},
     TITLE = {Vinberg pairs and {H}iggs bundles},
 BOOKTITLE = {Moduli spaces and vector bundles---new trends},
    SERIES = {Contemp. Math.},
    VOLUME = {803},
     PAGES = {199--222},
 PUBLISHER = {Amer. Math. Soc., [Providence], RI},
      YEAR = {[2024] \copyright 2024},
      ISBN = {978-1-4704-7296-2; [9781470476465]},
   MRCLASS = {14H60 (57R57 58D29)},
  MRNUMBER = {4773904},
       DOI = {10.1090/conm/803/16099},
       URL = {https://doi.org/10.1090/conm/803/16099},
}

@article{Garc_a_Prada_2026,
   title={Cyclic Higgs Bundles and the Toledo Invariant},
   ISSN={1531-586X},
   url={http://dx.doi.org/10.1007/s00031-026-09962-2},
   DOI={10.1007/s00031-026-09962-2},
   journal={Transformation Groups},
   publisher={Springer Science and Business Media LLC},
   author={García-Prada, Oscar and González, Miguel},
   year={2026},
   month=mar }

@misc{peónnieto2015cameraldatasup1phiggsbundles,
      title={Cameral data for SU(p+1,p)-Higgs bundles}, 
      author={Ana Peón-Nieto},
      year={2015},
      eprint={1506.01318},
      archivePrefix={arXiv},
      primaryClass={math.AG},
      url={https://arxiv.org/abs/1506.01318}, 
}

@article {Simpson94moduli1,
    AUTHOR = {Simpson, Carlos T.},
     TITLE = {Moduli of representations of the fundamental group of a smooth
              projective variety. {I}},
   JOURNAL = {Inst. Hautes \'{E}tudes Sci. Publ. Math.},
  FJOURNAL = {Institut des Hautes \'{E}tudes Scientifiques. Publications
              Math\'{e}matiques},
    NUMBER = {79},
      YEAR = {1994},
     PAGES = {47--129},
      ISSN = {0073-8301},
   MRCLASS = {14D20 (14D22 14D25 14F05)},
  MRNUMBER = {1307297},
MRREVIEWER = {Nitin Nitsure},
       URL = {http://www.numdam.org/item?id=PMIHES_1994__79__47_0},
}

@article {schmitt,
    AUTHOR = {Schmitt, Alexander},
     TITLE = {Moduli for decorated tuples of sheaves and representation
              spaces for quivers},
   JOURNAL = {Proc. Indian Acad. Sci. Math. Sci.},
  FJOURNAL = {Indian Academy of Sciences. Proceedings. Mathematical
              Sciences},
    VOLUME = {115},
      YEAR = {2005},
    NUMBER = {1},
     PAGES = {15--49},
      ISSN = {0253-4142,0973-7685},
   MRCLASS = {14D20 (14F05 16G20)},
  MRNUMBER = {2120597},
MRREVIEWER = {Alastair\ Craw},
       DOI = {10.1007/BF02829837},
       URL = {https://doi.org/10.1007/BF02829837},
}

@article{ngo2010lemme,
  title={Le lemme fondamental pour les algebres de Lie},
  author={Ng{\^o}, Bao Ch{\^a}u},
  journal={Publications Math{\'e}matiques de l'IH{\'E}S},
  volume={111},
  number={1},
  pages={1--169},
  year={2010},
  publisher={Springer}
}

@article{ngo2006fibration,
  title={Fibration de Hitchin et endoscopie},
  author={Ng{\^o}, Bao Ch{\^a}u},
  journal={Inventiones mathematicae},
  volume={164},
  number={2},
  pages={399--453},
  year={2006},
  publisher={Springer}
}

@article {hameister,
    AUTHOR = {Hameister, Thomas and Morrissey, Benedict},
     TITLE = {The {H}itchin fibration for symmetric pairs},
   JOURNAL = {Adv. Math.},
  FJOURNAL = {Advances in Mathematics},
    VOLUME = {482},
      YEAR = {2025},
     PAGES = {Paper No. 110560, 68},
      ISSN = {0001-8708,1090-2082},
   MRCLASS = {14D06 (14D23 14D24 14L24 17B45)},
  MRNUMBER = {4966334},
       DOI = {10.1016/j.aim.2025.110560},
       URL = {https://doi.org/10.1016/j.aim.2025.110560},
}

@misc{lee2025relativespectralcorrespondenceparabolic,
      title={Relative spectral correspondence for parabolic Higgs bundles and Deligne--Simpson problem}, 
      author={Jia Choon Lee and Sukjoo Lee},
      year={2025},
      eprint={2509.08527},
      archivePrefix={arXiv},
      primaryClass={math.AG},
      url={https://arxiv.org/abs/2509.08527}, 
}

@article{Diaconescu_2018,
	doi = {10.1007/s00220-018-3097-9},
  
	url = {https://doi.org/10.1007%2Fs00220-018-3097-9},
  
	year = 2018,
	month = {02},
  
	publisher = {Springer Science and Business Media {LLC}
},
  
	volume = {359},
  
	number = {3},
  
	pages = {1027--1078},
  
	author = {Duiliu-Emanuel Diaconescu and Ron Donagi and Tony Pantev},
  
	title = {{BPS} States, Torus Links and Wild Character Varieties},
  
	journal = {Communications in Mathematical Physics}
}

@article{Szab2017TheBG,
  title={The birational geometry of unramified irregular Higgs bundles on curves},
  author={Szil{\'a}rd Szab{\'o}},
  journal={International Journal of Mathematics},
  year={2017},
  volume={28},
  pages={1750045},
  url={https://api.semanticscholar.org/CorpusID:125343111}
}

@article {chuang,
    AUTHOR = {Chuang, Wu-yen and Diaconescu, Duiliu-Emanuel and Donagi, Ron
              and Pantev, Tony},
     TITLE = {Parabolic refined invariants and {M}acdonald polynomials},
   JOURNAL = {Comm. Math. Phys.},
  FJOURNAL = {Communications in Mathematical Physics},
    VOLUME = {335},
      YEAR = {2015},
    NUMBER = {3},
     PAGES = {1323--1379},
      ISSN = {0010-3616,1432-0916},
   MRCLASS = {81T30 (05E05 14N35)},
  MRNUMBER = {3320315},
       DOI = {10.1007/s00220-014-2184-9},
       URL = {https://doi.org/10.1007/s00220-014-2184-9},
}

@unpublished{KOP,
  author = {Katzarkov, Ludmil and Orlov, Dmitri and Pantev, Tony},
  title  = {Notes on Higgs bundles and D-branes},
  note   = {Unpublished notes},
}

@article {garcia-ramanan,
    AUTHOR = {Garc\'ia-Prada, Oscar and Ramanan, S.},
     TITLE = {Involutions and higher order automorphisms of {H}iggs bundle
              moduli spaces},
   JOURNAL = {Proc. Lond. Math. Soc. (3)},
  FJOURNAL = {Proceedings of the London Mathematical Society. Third Series},
    VOLUME = {119},
      YEAR = {2019},
    NUMBER = {3},
     PAGES = {681--732},
      ISSN = {0024-6115,1460-244X},
   MRCLASS = {14H60 (57R57 58D29)},
  MRNUMBER = {3960666},
MRREVIEWER = {Ronald\ A.\ Z\'u\~niga-Rojas},
       DOI = {10.1112/plms.12242},
       URL = {https://doi.org/10.1112/plms.12242},
}

\end{document}